\documentclass[preprint,12pt,number]{elsarticle}

\usepackage{amssymb}
\usepackage{mathrsfs}
\usepackage{hyperref}
\usepackage{amsmath}
\usepackage{amsthm}

\newtheorem{theorem}{Theorem}[section]
\newtheorem{lemma}[theorem]{Lemma}
\newtheorem{proposition}[theorem]{Proposition}

\newtheorem{remark}[theorem]{Remark}

\newtheorem{teo}{Theorem}[section]

\newtheorem{corollary}[teo]{Corollary}

\newcommand{\N}{\mathbb{N}}
\newcommand{\R}{\mathbb{R}}
\newcommand{\C}{\mathbb{C}}
\newcommand{\D}{\mathbb{D}}

\newcommand{\T}{\mathbb{T}}

\newcommand{\G}{\mathbb{G}}

\begin{document}

\begin{frontmatter}



\title{Interpolating Sequences For Dual Uniform Algebras} 


\author{Mario P. Maletzki} 

\affiliation{
            city={Teruel},
            country={Spain}}

\begin{abstract}
Given a dual uniform algebra $A=X^*$ with maximal ideal space $M_A$, we provide the first sufficient condition in terms of the Gleason distance of $A$ for a sequence in $M_A\cap X$ to be interpolating for $A$. We prove that a sequence in $\mathbb{D}^N$ is uniformly separated if and only if it is interpolating for $H^\infty(\mathbb{D}^N)$ and its sequence of norms satisfies the Blaschke condition, and then use this characterization to classify the interpolating sequences for $H^\infty(\mathbb{D}^N)$ in terms of its Gleason distance. We also study interpolating sequences for $\mathscr{H}^\infty$, the algebra of bounded Dirichlet series, obtaining necessary and sufficient conditions for a sequence in $\mathbb{C}+$ to be interpolating for this space, and relating the geometry of such sequences to that of the interpolating sequences for $H^\infty(\mathbb{C}+)$.  Finally, we show that a sequence in the Shilov boundary of the second dual of a uniform algebra $A$ is interpolating for $A^{**}$ if and only if it is discrete for the $w^*$-topology.
\end{abstract}



\end{frontmatter}



\section{Introduction}
\vspace{0.3cm}
Given a uniform algebra $A$ with maximal ideal space $M_A$, a sequence $(x_n)_n$ in $M_A$ is said to be interpolating for $A$ if for any bounded sequence of complex numbers $(\alpha_n)_n$ there exists $f \in A$ such that $f(x_n)=\alpha_n$ for every $n$. Interpolating sequences for uniform algebras, while of intrinsic interest, have proven to be useful in the study of composition operators \cite{GalindoGamelinLindstrom2004,Hosokawa2002}, and, as in the case of $H^\infty(\mathbb{D})$, they played a crucial role in describing the analytic structure of $\mathfrak{M}$, the maximal ideal space of $H^\infty(\mathbb{D})$ \cite{Hoffman1967}. The study of interpolating sequences for uniform algebras dates back to the seminal work of Carleson \cite{Carleson1958} in the 1950s, when he proved his celebrated interpolation theorem asserting that a sequence $(z_n)_n$ in $\mathbb{D}$ is interpolating for $H^\infty(\mathbb{D)}$ if and only if it satisfies \begin{equation}\label{Carleson's condition}
        \inf_{n\in\N}\prod_{k\neq n}\left| \frac{z_n-z_k}{1-\overline{z_n}z_k}\right|>0.
    \end{equation} Following Carleson's theorem, the uniform algebras for which interpolating sequences have been most extensively studied are the Hardy spaces $H^\infty(\Omega)$, consisting of bounded analytic functions on a bounded domain $\Omega \subset \mathbb{C}^n$ \cite{AglerMcCarthy2001,Amar1980,Berndtsson1985,BerndtssonChangLin1987,Kronstadt1974,KronstadtNeville1977}.

As a consequence of Carleson's theorem, we have that since the pseudohyperbolic distance $\rho_\mathbb{D}$ between two points $z,w\in \mathbb{D}$ is defined by $\rho_\mathbb{D}(z,w)=\left\vert{} \frac{z-w}{1-\overline{z}w}\right\vert{}$, the interpolating sequences for $H^\infty(\mathbb{D})$ are completely determined by the pseudohyperbolic distance between the points of the sequence. Considering then the pseudohyperbolic distance for a bounded domain $\Omega \subset \mathbb{C}^N$, which is defined as\begin{equation}\label{Dist-gleason}\rho_\Omega(z,w):=\sup\{|f(z)|:f\in H^\infty(\Omega),\ \lVert f \rVert\leq 1,\ f(w)=0 \},\end{equation}it seems natural to ask whether Carleson's theorem could be generalized to $H^\infty(\Omega)$. Berndtsson proved in \cite{Berndtsson1985} that the corresponding condition to \eqref{Carleson's condition} with the pseudohyperbolic distance $\rho_{\mathbb{B}^N}$ of the unit ball of $\mathbb{C}^N$ is sufficient---although no longer necessary when $N>1$---for a sequence in $\mathbb{B}^N$ to be interpolating for $H^\infty(\mathbb{B}^N)$. Moreover, together with Chang and Lin, he proved in \cite{BerndtssonChangLin1987} that the same condition with respect to the pseudohyperbolic distance in the polydisc is also sufficient for a sequence in $\mathbb{D}^N$ to be interpolating for $H^\infty(\mathbb{D}^N)$. These results led the authors to ask whether a uniformly separated sequence in a bounded domain $\Omega\subset\mathbb{C}^N$, that is, a sequence $(x_n)_n$ in $\Omega$ satisfying\begin{equation}\label{Carl-gen}\delta:= \inf_{n\in\mathbb{N}}\prod_{k\neq n}\rho_X(x_n,x_k)>0,\end{equation}is necessarily interpolating for $H^\infty(\Omega)$.

It is still unknown whether every uniformly separated sequence in a bounded domain $\Omega \subset \mathbb{C}^N$ must be interpolating for $H^\infty(\Omega)$, although it was proven in \cite{BerndtssonChangLin1987}  that a positive answer would follow if any sequence in $\mathbb{D}^N$ satisfying \eqref{Carl-gen} had a bound for its interpolation constant independent of $N$. Should such a uniform bound exist, it would imply that for any dual uniform algebra $A=X^*$ with Gleason distance $\rho_A$, every sequence $(x_n)_n$ in $M_A\cap X$ satisfying $\inf_{n\in\mathbb{N}}\prod_{k\neq n}\rho_A(x_n,x_k)>0$ is necessarily interpolating for $A$. This general interpolation problem, which was referred to as \emph{Carleson's generalized problem} in \cite{GalindoLindstromMiralles2009}, would thus carry profound implications due to the vast generality of this abstract framework. However, not only has this problem remained wide open, but no sufficient condition concerning the Gleason distance $\rho_A$ between the points of a sequence in $M_A\cap X$ to ensure interpolation for $A$ has been established so far.

One of the main goals of this paper is to provide the first such condition. In particular, considering for each $(x_n)_n$ in $M_A$ the sequence $\delta_n:=\prod_{k\neq n}\rho_A(x_n,x_k)$, we prove in Section \ref{Section-3} the following theorem:

\begin{theorem}\label{Main Theorem}Let $A=X^*$ be a dual uniform algebra and $(x_n)_n\subset M_A\cap X$ a sequence such that $\sum_{n=1}^\infty(1-\delta_n)^{1/3}<\infty$. Then $(x_n)_n$ is interpolating for $A$.\end{theorem}For a uniform algebra $A$ that is not necessarily a dual space, the theorem remains valid provided the interpolation is performed by functions in $A^{**}$. As a consequence, we recover as corollaries the theorems in \cite{CarneColeGamelin1989} and \cite{GalindoGamelinLindstrom2004} regarding the existence of interpolating sequences on sets that are not hyperbolically bounded.

Since the importance of interpolating sequences for $H^\infty(\mathbb{D}^N)$ was made evident in \cite{BerndtssonChangLin1987}, we continue its study in Section \ref{Section-polydisc}. We prove that an interpolating sequence for $H^\infty(\mathbb{D}^N)$ satisfying $\sum_{n=1}^\infty(1-\lVert z_n\rVert)<\infty$ must necessarily be uniformly separated. This yields a complete characterization of uniformly separated sequences in $\mathbb{D}^N$, which in turn allows us to classify the interpolating sequences for $H^\infty(\mathbb{D}^N)$ in terms of the Gleason distance.

In Section \ref{Section-4}, we consider the uniform algebra $\mathscr{H}^\infty$, consisting of bounded analytic functions $f$ in $\mathbb{C}_+=\{s:\text{Re}(s)>0\}$ that can be represented by a Dirichlet series$$f(s)=\sum_{n=1}^\infty \frac{a_n}{n^s}.$$The study of this algebra was initiated by H. Bohr, who considered the partial sums of the Riemann zeta function $\zeta(s)$ to study its zeros. Bohr observed that each function in $\mathscr{H}^\infty$ could be associated with a function in $H^\infty(B_{c_0})$, and it was proven in \cite{HedenmalmLindqvistSeip1997} that this association is in fact an isometric isomorphism. Using this connection with the Hardy space $H^\infty(B_{c_0})$, we provide a simple description of the Gleason distance for $\mathscr{H}^\infty$ and obtain new results on the interpolating sequences for this space. The first results on interpolating sequences for $\mathscr{H}^\infty$ appeared in \cite{seip2009interpolation}, where it was shown that among bounded sequences, the interpolating sequences for $\mathscr{H}^\infty$ are precisely the interpolating sequences for $H^\infty(\mathbb{C}_+)$. However, a complete geometric description of the interpolating sequences for $\mathscr{H}^\infty$, not necessarily bounded, remains a difficult problem. We use Theorem \ref{Main Theorem} to establish a general procedure for constructing unbounded interpolating sequences for $\mathscr{H}^\infty$, which reveals some surprising properties of the interpolation in this space. Moreover, using Kronecker's theorem on Diophantine approximation, we prove that for any sequence of positive real numbers $(\sigma_n)_n$ converging to $0$ such that $\limsup_n\frac{\sigma_{n+1}}{\sigma_n}=1$, there exists a sequence $(s_n)_n$ with $\text{Re}(s_n)=\sigma_n$ that is interpolating for $H^\infty(\mathbb{C}_+)$ but fails to be interpolating for $\mathscr{H}^{\infty}$. This highlights the different geometric natures of the interpolating sequences for $\mathscr{H}^{\infty}$ and $H^\infty(\mathbb{C}_+)$ in the unbounded setting.

In Section \ref{Section-5}, we focus on interpolating sequences for the second dual of a uniform algebra $A^{}$ that do not necessarily lie in $M_A$. Motivated by Hoffman's theorem asserting that a sequence in the Shilov boundary of $H^\infty(\mathbb{D})$ is interpolating if and only if it is discrete for the $w^*$-topology, we prove that $\partial_{A^{**}}$, the Shilov boundary of $A^{**}$, is always a hyperstonean space, and then show that for sequences in $\partial_{A^{**}}$ being interpolating for $A^{**}$ is equivalent to being discrete for the $w^*$-topology.

We conclude in Section \ref{Section-6} with a new approach to Carleson's generalized problem, which was the main motivation of this work, reducing the problem to a quantitative comparison between interpolation constants and Carleson intensities of sequences in polydiscs.

\section{Preliminaries}\label{Preliminaries}

\vspace{0.3cm}
\subsection{The Gleason distance}
\vspace{0.3cm}
Given a uniform algebra $A$, the Gleason distance between two elements $x$ and $y$ of $M_A$ is defined by 
\[
\rho_A(x,y):=\sup\{ |f(x)|: f\in A, \ \lVert f \rVert\leq 1, \ f(y)=0 \}.
\] 
The relation $x\sim y$ defined by $\rho_A(x,y)<1$ is an equivalence relation whose equivalence classes are called Gleason parts, and whenever a Gleason part consists of a single point, that point will be called trivial. It is not hard to prove that the Gleason distance satisfies 
\begin{equation}\label{Gleason-equality}
    \rho_A(x,y)=\sup\{ \rho_\D(f(x),f(y)): f\in A, \ \lVert f \rVert< 1 \},
\end{equation} 
and therefore that it is a generalization of Carathéodory's distance for a bounded domain $G$ of a Banach space. The second dual of $A$ is also a uniform algebra with any of the Arens products introduced in \cite{Arens1951}, which necessarily have to be equal, and the Gleason distance $\rho_{A^{**}}$ satisfies 
\[
\rho_{A^{**}}(x,y)=\rho_A(x,y)\quad \text{for every } x,y\in M_A.
\] 
The Shilov boundary of $A$, which is denoted by $\partial_A$, is the smallest closed subset of $M_A$ in which every function $f$ of $A$ attains its norm. It is well known that the Shilov boundary of $H^\infty(\D)$ is homeomorphic to the maximal ideal space of $L^\infty(\T)$, and moreover, that every point in its Shilov boundary is trivial (see \cite{Garnett1981}).

For a sequence $(x_n)_n$ in $M_A$, we define its associated interpolating operator $T:A\to \ell_\infty$ by $T(f):=(f(x_n))_n$ and its constant of interpolation by  $$ M:=\sup_{\lVert(\alpha_n)_n\rVert\leq1}\inf\{\lVert f \rVert:f\in A, \ f(x_n)=\alpha_n\ \mbox{for every } n\}.$$ Clearly, $(x_n)_n$ is interpolating for $A$ if and only if $T$ is surjective, and if $(x_n)_n$ is interpolating the Open Mapping theorem implies that its constant of interpolation if finite. It follows from \eqref{Gleason-equality}  and the Schwarz-Pick inequality that any interpolating sequence for $A$ must be separated  for the Gleason distance, that is, there exists a $\delta>0$ such that $\rho_A(x_n,x_k)\geq \delta$ whenever $n\neq k$. Moreover, if $S$ is the interpolating operator associated with another sequence $(y_n)_n$ of $M_A$, then 
\[
\lVert T-S \rVert=\sup_{\lVert f \rVert\leq 1}\{ \lVert (f(x_n)-f(y_n))_n \rVert \}\leq 2\sup_n\rho_A(x_n,y_n).
\] As a consequence, if $(x_n)_n$ is an interpolating sequence for $A$ with interpolating constant $M>0$, then any sequence $(y_n)_n$ in $M_A$ with $\sup_n\rho_A(x_n,y_n)<\frac{1}{2M}$ must also be interpolating for $A$. 

An effective formula for the Gleason distance is hardly available. However, if $\rho_{\D^N}$ denotes the Gleason distance of $H^\infty(\D^N)$, it is well known that if $z=(z_1,\ldots,z_N)$ and $w=(w_1,\ldots,w_N)$ belong to $\D^N$, then it is 
\[
\rho_{\D^N}(\delta_z,\delta_w)=\max_{1\leq i\leq N}\left| \frac{z_i-w_i}{1-\overline{z_i}w_i}\right|.
\] 
If $B_E$ denotes the unit ball of either a Hilbert space or a $C_0(X)$ space, the Gleason distance of $H^\infty(B_E)$ between the points of $B_E$ was given in \cite{AronGalindoLindstrom2003}, and if for $R>1$ we let $A_R$ be the annulus $\{z\in \C: R^{-1}<|z|<R\}$, then a computable formula for the Gleason distance for $H^\infty(A_R)$ between the points of $A_R$ follows from \eqref{Gleason-equality} and \cite{Simha1975}. We will now provide a simple description of the Gleason distance for two new uniform algebras, namely, for the algebra of bounded Dirichlet series in $\C_+$ and for the algebra of bounded analytic and symmetric functions in $\D^2$:

\subsubsection{Gleason distance for $\mathscr{H}^\infty$}

We first show that, as a consequence of \emph{Bohr's point of view}, the Gleason distance between two points $s_1$ and $s_2$ in $\C_+$ satisfies 
\begin{equation}\label{Gleason}
    \rho_{\mathscr{H}^{\infty}}(s_1,s_2)=\sup_{p}\{\rho_\D(p^{-s_1},p^{-s_2})\},
\end{equation} 
where the supremum is taken over the prime numbers. Bohr observed that if we factor each $n$ into its product of primes $n=p_1^{\alpha_1}\ldots p_r^{\alpha_r}$ (where $p_m$ denotes the $m$-th prime number) and we let $z_m:=p_m^{-s}$, then for any function $f(s)=\sum_{n=1}^\infty a_n n^{-s}$ in $\mathscr{H}^{\infty}$ we have that 
\[
f(s)=\sum_{n=1}^\infty a_n(p_1^{\alpha_1})^{-s}\ldots (p_r^{\alpha_r})^{-s} =\sum_{n=1}^\infty a_n z_1^{\alpha_1}\ldots z_r^{\alpha_r},
\] 
so we may identify each function in $\mathscr{H}^{\infty}$ with a function in $H^\infty(B_{c_0})$. Moreover, it was shown in \cite{HedenmalmLindqvistSeip1997} that this identification is in fact an isometric isomorphism, so in particular we have that $\rho_{\mathscr{H}^{\infty}}(s_1,s_2)=\rho_{c_0}((p_n^{-s_1})_n,(p_n^{-s_2})_n)$, where $\rho_{c_0}$ denotes the Gleason distance for $H^\infty(B_{c_0})$, and thus \eqref{Gleason} follows directly from \cite{AronGalindoLindstrom2003}.

Observe that if $s$ and $t$ are positive real numbers, then we simply have $\rho_{\mathscr{H}^{\infty}}(s,t)=\rho_\D(2^{-s},2^{-t})$, and moreover, since the pseudohyperbolic distance in $\D$ is invariant through rotations, it is $\rho_{\mathscr{H}^{\infty}}(s+iy,t+iy)=\rho_{\mathscr{H}^{\infty}}(s,t)$ for every real number $y$.

\subsubsection{Gleason for $H^\infty_s(\D^2)$}

We now consider for any $N>1$ the algebra $H^\infty_s(\D^N)$ consisting of those bounded analytic functions $f$ in $\D^N$ that are symmetric, that is, such that $f(z_1,\ldots,z_N)=f(z_{\pi(1)},\ldots,z_{\pi(N)})$ for every $(z_1,\ldots,z_N)$ in $\D^N$ and every permutation $\pi$ of $N$.

For its study, we first recall that given the map $\sigma_N=(\sigma_{N,1},\ldots,\sigma_{N,N}):\C^N\rightarrow\C^N$ where 
\[
\sigma_{N,k}(z_1,\ldots,z_n)=\sum_{1\leq j_1<\ldots<j_k\leq N}z_{j_1}\ldots z_{j_k}, \quad \text{for } 1\leq k\leq N,
\] 
the $N$-dimensional symmetrized polydisc is defined by $\G_N:=\sigma_N(\D^N)$. If we consider the operator $T:H^\infty(\G_N)\rightarrow H^\infty_s(\D^N)$ that maps $f\mapsto f\circ\sigma_N$, we easily see that $T$ is a well defined isometry. Moreover, it turns out that $T$ is also surjective. In fact, if $f$ belongs to $H^\infty_s(\D^N)$ and satisfies 
\[
f(z_1,\ldots,z_N)=\sum_{m=0}^\infty P_m(z_1,\ldots,z_N)\quad \text{for every } (z_1,\ldots,z_N)\in\D^N,
\]
uniqueness of Taylor's series implies that $P_m$ is a symmetric $m$-homogeneous polynomial for every $m$, and thus the Fundamental Theorem of Symmetric Polynomials guarantees the existence of a polynomial $p_m\in \C[z_1,\ldots,z_N]$ such that $p_m\circ \sigma_N=P_m$ for every $m$. It is then straightforward that the function $g:=\sum_{m=0}^\infty p_m$ belongs to $H^\infty(\G_N)$ and satisfies $T(g)=f$, so we conclude that $H^\infty_s(\D^N)$ is isometrically isomorphic to the Hardy space $H^\infty(\G_N)$.

As a consequence we have that $\rho_{H^\infty_s(\D^N)}(z,w)=\rho_{H^\infty(\G_N)}(\sigma_N(z),\sigma_N(w))$ for each $z,w\in \D^N$, and as for $N=2$ it follows from \cite{AglerYoung2004} that if $(s_1,p_1)=\sigma_2(z)$ and $(s_2,p_2)=\sigma_2(w)$, then 
\begin{equation}\label{Gleason-sym}
    \rho_{H^\infty_s(\D^2)}(z,w)=\sup_{\alpha\in\T}\left|\frac{(s_2p_1-s_1p_2)\alpha^2+2(p_2-p_1)\alpha+s_1-s_2}{(s_1-\overline{s_2}p_1)\alpha^2-2(1-p_1\overline{p_2})\alpha+\overline{s_2}-s_1\overline{p_2}}\right|.
\end{equation} 

For $N\geq3$, even though the geometry of $\G_N$ has been extensively studied and some estimates for the Carathéodory distance are known (see \cite{NikolovPflugThomasZwonek2008}), there is no computable formula as for $N=2$. 
\begin{remark}
It should be plain that considering any pair of symmetric polynomials $p(z_1,z_2)$ and $q(z_1,z_2)$ that generate $\C[z_1,z_2]$ and are algebraically independent, if we let $\Omega=\{(p(z_1,z_2),q(z_1,z_2)): z_1,z_2\in\D\}$, then a computable formula of the Gleason distance for $H^\infty(\Omega)$ follows directly from \eqref{Gleason-sym}.
\end{remark}
\vspace{0.3cm}
\subsection{Interpolating sequences for $H^\infty(\D)$} 
\vspace{0.3cm}
Sequences in $\D$ that are interpolating for $H^\infty(\D)$ are especially well-behaved, most notably because they are Blaschke sequences. Recall that a sequence $(z_n)_n$ in $\D$ is called Blaschke provided $\sum_{n=1}^\infty (1-|z_n|)<\infty$, and that for any such sequence we may consider its associated Blaschke product 
\[
B(z):=z^m\prod_{z_n\neq 0}\frac{-\overline{z_n}}{|z_n|}\frac{z-z_n}{1-\overline{z_n}z}\quad \text{for } z\in\D,
\] 
where $m=1$ if $z_n=0$ for some $n$ and $m=0$ otherwise, which is easily seen to satisfy $B(z)=0$ if and only if $z=z_n$ for some $n$. Since interpolating sequences for $H^\infty(\D)$ are Blaschke, adding a finite number of points to an interpolating sequence yields another interpolating sequence. It is also plain that interpolating sequences are invariant under automorphisms of $\D$, and as a consequence of Carleson's theorem also through conjugation. Furthermore, we have that any sequence that is eventually close enough to an interpolating sequence must also be interpolating:

\begin{lemma}\label{estab-int}
    If $(z_n)_n\subset \D$ is an interpolating sequence for $H^\infty(\D)$, then every sequence $(w_n)_n$ such that $\rho_\D(z_n,w_n)\rightarrow0$ is also interpolating. 
\end{lemma}
\begin{proof}
Let $M>0$ be the constant of interpolation of $(z_n)_n$, and $T_N$ and $S_N$ be the interpolating operators associated with the truncated sequences $(z_n)_{n=N}^\infty$ and $(w_n)_{n=N}^\infty$ respectively. Since the constant of interpolation of $(z_n)_{n=N}^\infty$ is trivially bounded by $M$, considering a positive integer $N$ such that $\rho_\D(z_n,w_n)<\frac{1}{2M}$ for each $n\geq N$ we have that 
\[
\lVert T_N-S_N \rVert=\sup_{\lVert f \rVert_\infty\leq1}\{ \lVert (f(z_n)-f(w_n))_{n=N}^\infty \rVert \}\leq 2 \sup_{n\geq N}\rho_\D(z_n,w_n)<\frac{1}{M},
\] 
so $S_N$ is surjective and in particular $(w_n)_n$ is interpolating for $H^\infty(\D)$.
\end{proof}

\begin{remark}
    Even though the previous lemma is false for a general uniform algebra $A$, it is straightforward to see that it also holds for $H^\infty(\Omega)$, where $\Omega$ is a proper simply connected domain of $\C$. It is also not hard to prove that the lemma is true for interpolating sequences for $\mathscr{H}^\infty$.
\end{remark}


For Blaschke sequences of positive real numbers, we will make use of two inequalities, which we state in the following lemma: 

\begin{lemma}\label{lema-blaschke}
    Consider a Blaschke sequence $(\delta_n)_n$ in $(0,1)$ and let $0<\alpha< \delta_n$ for every $n$. Then the following inequalities hold: 
    \[
    \prod_{n=1}^\infty \frac{\delta_n-\alpha}{1-\delta_n\alpha}\geq\frac{\left(\prod_{n=1}^\infty \delta_n\right)-\alpha}{1-\alpha\left(\prod_{n=1}^\infty \delta_n\right)},\quad \quad1-\prod_{n=1}^\infty\delta_n\leq \sum_{n=1}^\infty(1-\delta_n).
    \]
\end{lemma} 

One may prove that both inequalities hold for finite products by induction, and then the lemma follows by considering limits.

For the disk algebra $A(\D)$, which consists of those analytic functions on $\D$ which are continuous in $\overline{\D}$, a set $V\subset \overline{\D}$ is said to be interpolating provided every continuous function on $V$ is the restriction of a function in $A(\D)$. The next theorem, attributed to A. Beurling and W. Rudin and which is a special case of \cite{HeardWells1969}, characterizes the interpolating sets for $A(\D)$. For a proof see \cite{Wojtaszczyk1991}.


\begin{theorem}\label{Beurling-Rudin}
    Let $V\subseteq \overline{\D}$ be a closed subset. Then $A(\D)|_V=C(V)$ if and only if $m(V\cap \T)=0$ and $V\cap \D$ is an interpolating sequence for $H^\infty(\D)$.
\end{theorem}

Our interest in the disk algebra comes from the fact that given a function $g$ in a uniform algebra $A$ with $\lVert g \rVert\leq 1$ and any function $f$ of $A(\D)$, the composition $f\circ g$ belongs to $A$ and satisfies $\lVert f\circ g \rVert=\lVert f \rVert$. This follows from the fact that $A(\D)$ is the uniform closure of the polynomials in $\D$, and it can be easily seen that if $f$ only belongs to $H^\infty(\D)$ then the composition $f\circ g$ may no longer belong to $A$. In fact, given any uniform algebra $A$, if we consider a constant function $g\equiv c$ for some $c\in \T$ and a bounded analytic function in $\D$ that cannot be continuously extended to $c$, the composition is not even defined. However, if $g\in A$ satisfies $\lVert g \rVert<1$ then the composition with functions from $H^\infty(\D)$ is well defined and belongs to $A$. To see this observe that if $f(z)=\sum_{n=0}^\infty\frac{f^{(n)}(0)}{n!}z^n$ for every $z\in\D$, then Cauchy's inequalities yield 
\[
\sum_{n=0}^\infty \left\lVert \frac{f^{(n)}(0)}{n!}g^n \right\rVert \leq \lVert f \rVert\sum_{n=0}^\infty \lVert g^n \rVert=\lVert f \rVert\sum_{n=0}^\infty \lVert g \rVert^n<\infty,
\] 
and thus by completeness of $A$ we have that $f\circ g=\sum_{n=0}^\infty \frac{f^{(n)}(0)}{n!}g^n\in A$. 

A key result in the proof of Theorem \ref{Main Theorem} will be the following one due to Naftalevi\v{c}: 

\begin{theorem}
    Given a sequence of real numbers $(r_n)_n$ in $(0,1)$ satisfying $\sum_{n=1}^\infty(1-r_n)<\infty$, there is a sequence $(z_n)_n\subset \D$ which is interpolating for $H^\infty(\D)$ and satisfies $|z_n|=r_n$ for every $n$.
\end{theorem}

In the proof of Naftalevi\v{c}'s theorem\footnote{Since the original paper from Naftalevi\v{c} seems not to be easily available, the reader is referred to \cite{Cochran1990} for a proof.}, it is shown that after possibly considering a larger sequence $(\alpha_n)_n$ that is also Blaschke and then truncating the sequence so that it satisfies $\sum_{n=1}^\infty(1-\alpha_n)<1/4$, letting 
\[
\theta_n:=2\pi\sum_{k=1}^{n-1}(1-\alpha_k),
\]
the sequence defined by $z_n:=\alpha_ne^{i\theta_n}$ satisfies \eqref{Carleson's condition}, and thus the subsequence of $(z_n)_n$ with indices corresponding to $(r_n)_n$ is interpolating for $H^\infty(\D)$. 

Since interpolating sequences for $H^\infty(\D)$ are invariant through rotations and conjugation, it is clear that if we let 
\[
\theta_n':=2\pi\sum_{k=n}^{\infty}(1-\alpha_k)\quad \text{for each } n,
\] 
then the sequence $(z_n')_n$ defined by $z_n':=\alpha_ne^{i\theta_n'}$ would also be interpolating for $H^\infty(\D)$, and moreover, we would have that $(\arg(z_n'))_n$ is a decreasing sequence of strictly positive real numbers converging to $0$. 

We will be interested in determining when the interpolating sequence $(z_n')_n$ previously defined approaches $1$ in such a way that for any horocycle tangent to $1$, $(z_n')_n$ eventually lies in that horocycle. The following lemma gives a sufficient condition for this to happen:

\begin{lemma}\label{lema-naft}
    Suppose that $(r_n)_n$ is a non-decreasing sequence in $(0,1)$ satisfying $\sum_{n=1}^\infty\sqrt{1-r_n}<\infty$ and we let $\theta_n:=2\pi\sum_{k=n}^{\infty}(1-r_k)$ for each $n$. Then the sequence defined by $z_n:=r_ne^{i\theta_n}$ satisfies $\lim_{n\rightarrow\infty}\frac{1-|z_n|^2}{|1-z_n|^2}=\infty$.
\end{lemma}
\begin{proof}
Since $|1-z_n|^2=1-2r_n\cos(\theta_n)+r_n^2=(1-r_n)^2+4r_n\sin^2(\frac{\theta_n}{2})$, we have that 
\[
\frac{1-|z_n|^2}{|1-z_n|^2}=\frac{1-r_n^2}{(1-r_n)^2+4r_n\sin^2(\frac{\theta_n}{2})}=\frac{1+r_n}{(1-r_n)+\frac{4r_n\sin^2(\frac{\theta_n}{2})}{1-r_n}},
\] 
and therefore $\lim_{n\rightarrow\infty}\frac{1-|z_n|^2}{|1-z_n|^2}=\infty$ holds if and only if $\lim_{n\rightarrow\infty}\frac{\sin^2(\frac{\theta_n}{2})}{1-r_n}=0$. Now, $\theta_n\rightarrow0$ implies that $\lim_{n\rightarrow\infty}\frac{\sin^2(\frac{\theta_n}{2})}{1-r_n}=\lim_{n\rightarrow\infty}\frac{\pi^2 \left(\sum_{k=n}^\infty (1-r_k)\right)^2}{1-r_n}$, and since $(r_n)_n$ is non-decreasing we have for each $n$ that $\sum_{k=n}^\infty (1-r_k)\leq \sqrt{1-r_n}\left(\sum_{k=n}^\infty \sqrt{1-r_k}\right)$, and therefore 
\[
\lim_{n\rightarrow\infty}\frac{\sin^2(\frac{\theta_n}{2})}{1-r_n}=\lim_{n\rightarrow\infty}\frac{\pi^2 \left(\sum_{k=n}^\infty (1-r_k)\right)^2}{1-r_n}\leq \lim_{n\rightarrow\infty} \left(\sum_{k=n}^\infty \sqrt{1-r_k}\right)^2=0.
\]
\end{proof}

\section{Proof of Theorem \ref{Main Theorem}}\label{Section-3}
\vspace{0.3cm}


Given a sequence $(z_n)_n$ in $\D$, we have that $\delta_{n_0}=\prod_{k\neq n_0}\rho_\D(z_{n_0},z_k)>0$ for some $n_0$ if and only if $(z_n)_n$ is a Blaschke sequence, and thus $\delta_n>0$ for every $n$ provided $\delta_{n_0}>0$ for some $n_0$. The next lemma, which will be needed for the proof of Theorem \ref{Main Theorem},  shows that the same holds for any sequence in $M_A$:

\begin{lemma}\label{lema-tec}
    Given a uniform algebra $A$ and a sequence $(x_n)_n$ in $M_A$, if $\delta_{n_0}>0$ for some $n_0$, then $\delta_n>0$ for every $n$.
\end{lemma}

\begin{proof}
    If there were only finitely many points in each different Gleason part the lemma would be trivial, and it is straightforward that we only need to prove the case when all the sequence lies in the same Gleason part. 
   
   Given $m\neq n_0$, the hypothesis $\delta_{n_0}>0$ implies for some $N>\max\{n_0,m\}$ that $\prod_{n=N}^\infty \rho(x_n,x_m)> \rho(x_{n_0},x_m)$, and then we have by Lemma \ref{lema-blaschke} that 
   \[
   \begin{aligned}
   \prod_{n=N}^\infty \rho(x_n,x_m) &\geq \prod_{n=N}^\infty \frac{\rho(x_{n},x_{n_0})-\rho(x_{n_0},x_m)}{1-\rho(x_{n},x_{n_0})\rho(x_{n_0},x_m)} \\
   &\geq \frac{\left(\prod_{n=N}^\infty\rho(x_{n},x_{n_0})\right)-\rho(x_{n_0},x_m)}{1-\left(\prod_{n=N}^\infty\rho(x_{n},x_{n_0})\right)\rho(x_{n_0},x_m)}>0,
   \end{aligned}
   \] 
   and thus we have 
   \[
   \delta_m=\left(\prod_{\substack{n<N\\n\neq m}}\rho(x_{n},x_m)\right)\left(\prod_{n\geq N}\rho(x_{n},x_m)\right)>0.
   \]
   
\end{proof}


We are now ready to prove the main theorem: 

\begin{proof}[Proof of Theorem \ref{Main Theorem}]
    We first observe that considering any permutation $(x_{\pi(n)})_n$ of the sequence $(x_n)_n$, if we let $\delta_n':=\prod_{k\neq \pi(n)}\rho_A(x_k,x_{\pi(n)})$ we would have $\delta_n'=\delta_{\pi(n)}$, so without loss of generality  we may assume that $(\delta_n)_n$ is a non-decreasing sequence. By Lemma \ref{lema-tec} we have that $\delta_n>0$ for every $n$, and since the sequence $(x_n)_n$ lies in $M_A\cap X$, Alaoglu's theorem implies that for each $n$ there is a function $f_n$ in $A$ with $\lVert f_n \rVert\leq1$ such that $f_n(x_n)=\delta_n$ while $f_n(x_k)=0$ for each $k\neq n$. Moreover, considering another non-decreasing sequence $(\mu_n)_n$ with $0<\mu_n<\delta_n$ and $\sum_{n=1}^\infty(1-\mu_n)^{1/3}<\infty$, we may assume that $\lVert f_n \rVert<1$ and  $f_n(x_n)=\mu_n$ for every $n$. 
    
    If we define $r_n:=1-(1-\mu_n)^{1/3}$ for every $n$, we have that the sequence $\left(\frac{\mu_n-r_n}{1-r_n\mu_n}\right)_n$ is also non-decreasing, and moreover it satisfies

    $$\sum_{n=1}^\infty\sqrt{1-\left(\frac{\mu_n-r_n}{1-r_n\mu_n}\right)} = \sum_{n=1}^\infty\sqrt{\frac{(1-\mu_n)(1+r_n)}{1-r_n\mu_n}} $$$$ \leq 2 \sum_{n=1}^\infty\sqrt{\frac{(1-\mu_n)}{(1-\mu_n)+(1-r_n)}} 
    = 2 \sum_{n=1}^\infty\sqrt{\frac{(1-\mu_n)}{(1-\mu_n)+(1-\mu_n)^{1/3}}} $$ $$ = 2 \sum_{n=1}^\infty\frac{(1-\mu_n)^{1/3}}{\sqrt{(1-\mu_n)^{2/3}+1}}<\infty,$$

    so it follows from the discussion on Naftalevi\v{c}'s theorem and Lemma \ref{lema-naft} that if we let $\sigma_n:=2\pi\sum_{k=n}^\infty \left(1-\left(\frac{\mu_k-r_k}{1-r_k\mu_k}\right)\right)$, then the sequence $(z_n)_n$ defined by $z_n:=\left(\frac{\mu_n-r_n}{1-r_n\mu_n}\right)e^{i\sigma_n}$ is interpolating for $H^\infty(\D)$ and satisfies that $\lim_{n\rightarrow\infty}\frac{1-|z_n|^2}{|1-z_n|^2}=\infty$. 

    Considering the Möbius transformations $\varphi_n(z)=\frac{z-r_ne^{i\sigma_n}}{1-r_ne^{-i\sigma_n}z}$ and the functions $h_n:=e^{i\sigma_n}f_n$, the compositions $g_n:=\varphi_n\circ h_n$ are well defined, belong to $A$, and satisfy $\lVert g_n\rVert < 1$ for each $n$. Consequently, for each positive integer $N$, the function 
    \[
    G_N:=\frac{\left(\sum_{n=1}^N\frac{1+g_n}{1-g_n}\right)-1}{1+\left(\sum_{n=1}^N\frac{1+g_n}{1-g_n}\right)},
    \] 
    also belongs to $A$ and satisfies $\lVert G_N\rVert \leq 1$. By Alaoglu's theorem the sequence $(G_N)_N$ has a subnet $(G_{N_i})_{i\in I}$ that is $w^*$-convergent to some $G\in A$ with $\lVert G \rVert\leq 1$, and since $(x_n)_n$ lies in $M_A\cap X$ we have that 
    \[
    G(x_n)=\lim_i \frac{\left(\sum_{k=1}^{N_i}\frac{1+g_k(x_n)}{1-g_k(x_n)}\right)-1}{1+\left(\sum_{k=1}^{N_i}\frac{1+g_k(x_n)}{1-g_k(x_n)}\right)}=\frac{\left(\sum_{k=1}^\infty\frac{1+g_k(x_n)}{1-g_k(x_n)}\right)-1}{1+\left(\sum_{k=1}^\infty\frac{1+g_k(x_n)}{1-g_k(x_n)}\right)}.
    \]
   
   We claim that $(G(x_n))_n$ is an interpolating sequence for $H^\infty(\D)$ such that $G(x_n)\rightarrow1$, and since the mapping $\psi:\C_+\rightarrow\D$ defined by $w\mapsto\frac{w-1}{1+w}$ is a conformal mapping, we just need to prove that the sequence $(w_n)_n$ defined by $w_n:=\sum_{k=1}^\infty\frac{1+g_k(x_n)}{1-g_k(x_n)}$ is interpolating for $H^\infty(\C_+)$. 
   
   First we observe that since the sequence $(z_n)_n$ previously defined is interpolating for $H^\infty(\D)$ and for each $n$ we have that 
   \[
   g_n(x_n)=\varphi_n(e^{i\sigma_n}\mu_n)=\frac{\mu_ne^{i\sigma_n}-r_ne^{i\sigma_n}}{1-r_n\mu_n}=\frac{\mu_n-r_n}{1-r_n\mu_n}e^{i\sigma_n}=z_n,
   \] 
   it turns out that $\left(\frac{1+g_n(x_n)}{1-g_n(x_n)}\right)_n$ is interpolating for $H^\infty(\C_+)$. Moreover, since $z_n\rightarrow1$ the claim will follow once we prove that $\lim_{n\rightarrow\infty}\rho_{\C_+}\left(w_n,\frac{1+g_n(x_n)}{1-g_n(x_n)}\right)=0$, as Lemma \ref{estab-int} would imply that $(w_n)_n$ is also interpolating for $H^\infty(\C_+)$, and by the Schwarz-Pick lemma we would have that 
   \[
   |G(x_n)- z_n|=\left|\psi(w_n)-\psi\left(\frac{1+g_n(x_n)}{1-g_n(x_n)}\right)\right|\leq 2\rho_{\C_+}\left(w_n,\frac{1+g_n(x_n)}{1-g_n(x_n)}\right)\rightarrow0,
   \] 
   and thus that $G(x_n)\rightarrow1$. 
   
   Since $\frac{1-|z_n|^2}{|1-z_n|^2}\rightarrow\infty$ and an easy computation shows that 
   \[
   \begin{aligned}
   \rho_{\C_+}\left(w_n,\frac{1+g_n(x_n)}{1-g_n(x_n)}\right) &=\left|\frac{\sum_{k\neq n}\frac{1+g_k(x_n)}{1-g_k(x_n)}}{2\text{Re}\left(\frac{1+g_n(x_n)}{1-g_n(x_n)}\right) +\sum_{k\neq n}\frac{1+g_k(x_n)}{1-g_k(x_n)}} \right| \\
   &=\left|\frac{\sum_{k\neq n}\frac{1+g_k(x_n)}{1-g_k(x_n)}}{2\frac{1-|z_n|^2}{|1-z_n|^2} +\sum_{k\neq n}\frac{1+g_k(x_n)}{1-g_k(x_n)}} \right|,
   \end{aligned}
   \] 
   we have that in order to prove $\rho_{\C_+}\left(w_n,\frac{1+g_n(x_n)}{1-g_n(x_n)}\right)\rightarrow0$ we only need to show that $\sup_n\left|\sum_{k\neq n}\frac{1+g_k(x_n)}{1-g_k(x_n)}\right|<\infty$. Now, taking into account that $g_n(x_k)=\varphi_n(0)=-r_ne^{i\sigma_n}$ for each $k\neq n$, we have that $$\left|\sum_{k\neq n}\frac{1+g_k(x_n)}{1-g_k(x_n)}\right|\leq \sum_{k=1}^\infty \left|\frac{1-r_ke^{i\sigma_k}}{1+r_ke^{i\sigma_k}}\right|,$$ and since $r_ke^{i\sigma_k}\rightarrow1$, the last series is convergent if and only if $\sum_{k=1}^\infty|1-r_ke^{i\sigma_k}|<\infty$. From the estimate 
   \[
   \begin{aligned}
   |1-r_ke^{i\sigma_k}| &=\sqrt{1-2r_k\cos\left(\frac{\sigma_k}{2}\right)+r_k^2}=\sqrt{(1-r_k)^2+4r_k\sin^2\left(\frac{\sigma_k}{2}\right)} \\
   &\leq (1-r_k)+2\sqrt{r_k}\sin\left(\frac{\sigma_k}{2}\right)=(1-\mu_k)^{1/3}+2\sqrt{r_k}\sin\left(\frac{\sigma_k}{2}\right)
   \end{aligned}
   \] 
   and the hypothesis $\sum_{n=1}^\infty(1-\mu_n)^{1/3}<\infty$, the problem is then reduced to the convergence of $\sum_{k=1}^\infty \sin\left(\frac{\sigma_k}{2}\right)$, or equivalently, to that of $\sum_{k=1}^\infty \sigma_k$. However, since $((1-\mu_n)^{1/3})_n$ is a non-increasing sequence of positive real numbers, we have that 
   \[
   \begin{aligned}
   \sum_{k=1}^\infty \sigma_k &=\sum_{k=1}^\infty\left(2\pi\sum_{n=k}^\infty\frac{(1-\mu_n)(1+r_n)}{1-r_n\mu_n}\right)=2\pi\sum_{n=1}^\infty n\frac{(1-\mu_n)(1+r_n)}{1-r_n\mu_n} \\
   &\leq 8\pi\sum_{n=1}^\infty n\frac{(1-\mu_n)}{(1-\mu_n)+(1-r_n)}\leq 8\pi\sum_{n=1}^\infty n(1-\mu_n)^{2/3} \\
   &=8\pi\sum_{k=1}^\infty\left(\sum_{n=2^k-1}^{2^{k+1}} n(1-\mu_n)^{2/3}\right)\leq 8\pi\sum_{k=1}^\infty 2^k\left( 2^{k+1} (1-\mu_{2^k-1})^{2/3}\right) \\
   &\leq 64\pi\sum_{k=1}^\infty\left( (2^{k}-1) (1-\mu_{2^k-1})^{1/3}\right)^2<\infty,
   \end{aligned}
   \] 
   and therefore we conclude that $\rho_{\C_+}\left(w_n,\frac{1+g_n(x_n)}{1-g_n(x_n)}\right)\rightarrow0$, proving the claim. 

    By Theorem \ref{Beurling-Rudin} we have that the set $(G(x_n))_n\cup \{1\}$ is interpolating for $A(\D)$, and therefore given any bounded sequence $(\alpha_n)_n\in \ell_\infty$ there is a constant $M>0$ such that if we let 
    \[
    \alpha_n^k:=\left\{\begin{array}{cc}
       \alpha_n  & \text{if } n\leq k, \\
       0  & \text{if } n> k,
    \end{array}\right.
    \] 
    then there is a function $F_k\in A(\D)$ with $F_k(G(x_n))=\alpha_n^k$ for every $n$ and $\lVert F_k \rVert<M$. Since the composition $F_k\circ G$ belongs to $A$ and satisfies $\lVert F_k\circ G \rVert<M$ for every $k$, letting $F$ be the $w^*$-limit of some subnet of $(F_k\circ G)_k$ we have that $F(x_n)=\alpha_n$ for every $n$, from where we conclude that $(x_n)_n$ is an interpolating sequence for $A$.

\end{proof}

If we consider sequences in $M_A$, it is easily verified that the same proof verbatim gives the following theorem:
\begin{theorem}\label{Main-teo-bidual}
    Given a uniform algebra $A$, if a sequence $(x_n)_n\subset M_A$ satisfies $\sum_{n=1}^\infty(1-\delta_n)^{1/3}<\infty$, then $(x_n)_n$ is interpolating for $A^{**}$.
\end{theorem}

However, the last theorem may fail if we do not allow the interpolation to be done with functions in $A^{**}$. Furthermore, Theorem \ref{Main Theorem} would fail if we only required the sequence be in $M_A$. For instance, considering the dual uniform algebra $H^\infty(\D)$, R. Mortini proved in \cite{Mortini2000} that a sequence of trivial points in $\mathfrak{M}$ is interpolating if and only if it is discrete for the $w^*$-topology, so any non-discrete sequence $(x_n)_n$ of trivial points in $\mathfrak{M}$ would satisfy that $\delta_n=1$ for every $n$ while not being interpolating for $H^\infty(\D)$. Moreover,  we note that the assumption of $A$ being a uniform algebra is essential. If we consider a commutative Banach algebra with identity $B$ rather than a uniform algebra, the problem at generalizing Theorem \ref{Main Theorem} stems from the fact that even if the composition of elements from the algebra $B$ with polynomials are well defined,  the norm of the composition is  no longer equal to the norm of the polynomial as a function in the disk algebra. In fact, if the norms were to be the same, considering the function $p(z)=z^2$ we would have that $\lVert f^2 \rVert=\lVert f \rVert^2$ for any element $f$ of $B$, forcing $B$ to be a uniform algebra. 

In \cite{GalindoGamelinLindstrom2004}, the authors defined a set $E\subset M_A$ to be hyperbolically bounded provided it is contained in a finite union of Gleason balls of radii strictly smaller than $1$. From Bear's formula of the Gleason distance between two points $x,y\in M_A$, that is
\[
\rho_A(x,y)=\frac{\lVert x-y \rVert}{4-\lVert x-y \rVert^2}, 
\]
it is clear that a set $E$ is hyperbolically bounded if and only if $\sup_{x,y\in E}\lVert x-y \rVert<2$. From Theorem \ref{Main-teo-bidual} it is clear that any set $E\subset M_A$ that is not hyperbolically bounded contains an interpolating sequence for $A^{**}$, thus obtaining Theorem 5.5 from \cite{CarneColeGamelin1989}. Also, the same argument with Theorem \ref{Main Theorem} shows that any sequence in $M_A\cap X$ that is not hyperbolically bounded contains an interpolating sequence for $A=X^*$ (cf. \cite[Proposition 4.2]{GalindoLindstromMiralles2009}).

\begin{remark}
    In the proof of Theorem \ref{Main Theorem}, we constructed a function $G$ in $A$ with $\lVert G \rVert\leq1$ such that $(G(x_n))_n$ is interpolating for $H^\infty(\D)$ in order to prove that $(x_n)_n$ is interpolating for $A$. However, there are interpolating sequences for dual uniform algebras for which this is not possible. To see this, observe that if such a function $G$ exists for a sequence $(x_n)_n\subset M_A\cap X$, then by Carleson's theorem and the inequality $\rho_\D(G(x),G(y))\leq \rho_A(x,y)$ we would have that $\inf_{n}\prod_{k\neq n}\rho_A(x_n,x_k)>0$, but it was shown in \cite{BerndtssonChangLin1987} that a sequence $(x_n)_n\subset \D^2$ may be interpolating for $H^\infty(\D^2)$ and satisfies $\prod_{k\neq n}\rho_{\D^2}(x_n,x_k)=0$ for every $n$.
\end{remark}

\section{Interpolating sequences in the polydisc}\label{Section-polydisc}
\vspace{0.3cm}
In this section, we continue the study of interpolating sequences in the polydisc initiated by Berndtsson, Chang, and Lin \cite{BerndtssonChangLin1987}. They showed that any sequence $(z_n)_n$ in $\D^N$ satisfying $\delta:=\inf_n \prod_{k\neq n}\rho_{\D^N}(z_k,z_n)>0$ is interpolating for $H^\infty(\D^N)$, but also  that if $N>1$ then this condition is no longer necessary for $(z_n)_n$ to be interpolating. In fact, they constructed an example of an interpolating sequence $(p_n)_n$ in $\D^2$ for which $\prod_{k\neq n}\rho_{\D^2}(p_k,p_n)=0$ for every $n$. It is therefore natural to ask what additional condition must an interpolating sequence satisfy for it to be uniformly separated,  and it turns out that this exact condition is that the sequence of norms be Blaschke: 



\begin{theorem}
    A sequence $(z_n)_n$ in $\D^N$ satisfies 
    \[
    \delta:=\inf_n \prod_{k\neq n}\rho_{\D^N}(z_k,z_n)>0
    \]
    if and only if $(z_n)_n$ is an interpolating for $H^\infty(\D^N)$ such that $\sum_{n=1}^\infty(1-\lVert z_n\rVert)<\infty$.
\end{theorem}

\begin{proof}
For each positive integer $n$ we consider the partition of $\N\setminus\{n\}$ into the sets $s_n(1), s_n(2), \dots, s_n(N)$ defined by letting $k$ be in $s_n(1)$ if $\rho_{\D^N}(z_n,z_k)=\rho_{\D}(z_n^1,z_k^1)$, and for $2\leq i\leq N$, letting $k$ be in $s_n(i)$ if $k\notin s_n(j)$ for $j<i$ and $\rho_{\D^N}(z_n,z_k)=\rho_{\D}(z_n^i,z_k^i)$. Using these partitions it is then easily seen that $\prod_{k\neq n}\rho_{\D^N}(z_k,z_n)>0$ for some $n$ if and only if $\sum_{n=1}^\infty(1-\lVert z_n\rVert)<\infty$. 

Now, if $(z_n)_n$ is interpolating, then by the Open Mapping theorem there are functions $(f_n)_n\subset H^\infty(\D^N)$ and a constant $M>0$ such that $f_n(z_k)=\delta_{n,k}$ for every pair of positive integers $k,n$ while $\lVert f_n \rVert<M$ for every $n$. Fixing a positive integer $m>1$, we then define for each $n$ the functions 
\[
\phi_n^m(z^1,\dots,z^N):=f_n(z^1,\dots,z^N)\prod_{i=1}^N\left(\prod_{\substack{k=1\\k\in s_n(i)}}^m \frac{1-\overline{z_k^i}z^i}{z^i-z_k^i} \right).
\] 
Each $\phi_n^m$ is holomorphic on $\D^N\setminus E_n$, where 
\[
E_n:=\bigcup_{i=1}^N\left(\bigcup_{\substack{k=1\\k\in s_n(i)}}^m\left(\D^{i-1}\times\{z_k^i\}\times \D^{N-i}\right)\right),
\] 
and we claim there is a holomorphic extension of $\phi_n^m$ to $\D^N$ that is bounded by $M$. In fact, if we consider $z_k^i$ for some $1\leq i\leq N$ and $k\in s_n(i)$ with $1\leq k\leq m$, then for every $(w^1,\dots,w^{i-1},z_k^i,w^{i+1},\dots,w^N)$ in $E_n$ such that no $w^j$ belongs to the set $\{z_k^j : k\in s_n(j), 1\leq k\leq m\}$, we have that $\varphi_{n,i}^m(z):=\phi_n^m(w^1,\dots,w^{i-1},z,w^{i+1},\dots,w^N)$ has removable singularities in $\{z_k^i : k\in s_n(i), 1\leq k\leq m\}$. Denoting also by $\varphi_{n,i}^m$ its holomorphic extension to $\D$, we define $\widetilde{\phi_n^m}=\phi_n^m$ outside $E_n$ and $\widetilde{\phi_n^m}(w^1,\dots,w^{i-1},z_k^i,w^{i+1},\dots,w^N):=\varphi_{n,i}^m(z_k^i)$. If we now consider a point $(w^1,\dots,w^{i-1},z_k^i,w^{i+1},\dots,w^N)$ such that there is exactly one $w^j$ belonging to $\{z_k^j : k\in s_n(j), 1\leq k\leq m\}$, then the function that maps 
\[
z\mapsto\widetilde{\phi_n^m}(w^1,\dots,w^{j-1},z,w^{j+1},\dots,w^{i-1},z_k^i,w^{i+1},\dots,w^N)
\] 
also has removable singularities in $\{z_k^j : k\in s_n(j), 1\leq k\leq m\}$, and thus we may define $\widetilde{\phi_n^m}$ in these points as the values of its holomorphic extension. Repeating the argument $N-2$ more times, we will have recursively extended the function $\phi_n^m$ to a holomorphic function $\widetilde{\phi_n^m}$ on $\D^N$. Moreover, by the Maximum Modulus Principle for $\D^N$ we have that $\lVert\widetilde{\phi_n^m}\rVert=\lim_{r\to 1}\max\{|\widetilde{\phi_n^m}(z^1,\dots,z^N)|: (z^i)_{i=1}^N\subset r\T\}$, so in particular it is $\lVert\widetilde{\phi_n^m}\rVert=\lVert f_n\rVert\leq M$. 
    
The condition $\lVert\widetilde{\phi_n^m}\rVert\leq M$ guarantees that 
\[
|f_n(z^1,\dots,z^N)|\leq M \left(\prod_{i=1}^N\left(\prod_{\substack{k=1\\k\in s_n(i)}}^m \left|\frac{z^i-z_k^i}{1-\overline{z_k^i}z^i}\right| \right)\right)\quad \text{for }\ (z^1,\dots, z^N)\in\D^N,
\] 
and in particular 
\[
1=|f_n(z_n^1,\dots,z_n^N)|\leq M \left(\prod_{i=1}^N\left(\prod_{\substack{k=1\\k\in s_n(i)}}^m \left|\frac{z_n^i-z_k^i}{1-\overline{z_k^i}z_n^i}\right| \right)\right).
\] 

Taking the limit as $m \to \infty$, the last product converges because of the hypothesis $\sum_{n=1}^\infty(1-\lVert z_n\rVert)<\infty$, and we conclude that 
\[
\prod_{\substack{k=1\\k\neq n}}^\infty \rho_{\D^N}(z_n,z_k)=\prod_{i=1}^N\left(\prod_{\substack{k=1\\k\in s_n(i)}}^\infty \left|\frac{z_n^i-z_k^i}{1-\overline{z_k^i}z_n^i}\right| \right)\geq \frac{1}{M}>0\quad \text{for every } n.
\]
\end{proof}

As a consequence, we have the following dichotomy for interpolating sequences for $H^\infty(\D^N)$ in terms of the Gleason distance:

\begin{corollary}
    If $(z_n)_n$ is an interpolating sequence for $H^\infty(\D^N)$, then either there is some $\delta>0$ such that $\inf_n \prod_{k\neq n}\rho_{\D^N}(z_k,z_n)\geq \delta $, or it is $\prod_{k\neq n}\rho_{\D^N}(z_k,z_n)=0$ for every $n$.
\end{corollary}

Observe that considering a sequence $(z_n)_n$ in $B_{c_0}$, we could similarly define a partition of $\N\setminus\{n\}$ into the sets $(s_n(i))_{i=1}^\infty$ so that 
\[
\prod_{k\neq n}\rho_{c_0}(z_k,z_n)=\prod_{i=1}^\infty\left(\prod_{\substack{k=1\\k\in s_n(i)}}^\infty \left|\frac{z_n^i-z_k^i}{1-\overline{z_k^i}z_n^i}\right| \right).
\] 
Moreover, if $(z_n)_n$ was interpolating for $H^\infty(B_{c_0})$ and $(f_n)_n$ were functions in $H^\infty(B_{c_0})$ with $f_n(z_k)=\delta_{n,k}$ and $\lVert f_n \rVert<M$ for every $n$, defining for each $N,m>1$ the function 
\[
\phi_n^{N,m}(z^1,\dots,z^N):=f_n(z^1,\dots,z^N,0,0,\dots)\prod_{i=1}^N\left(\prod_{\substack{k=1\\k\in s_n(i)}}^m \frac{1-\overline{z_k^i}z^i}{z^i-z_k^i} \right),
\] 
we would similarly have that each $\phi_n^{N,m}$ has a holomorphic extension $\widetilde{\phi_n^{N,m}}$ that belongs to $H^\infty(\D^N)$ and is bounded by $M$. We would then have for each $m>1$ that 
\[
1=|f_n(z_n)|=\lim_{N\to\infty}f_n(z^1,\dots,z^N,0,0,\dots)\leq M \left(\prod_{i=1}^\infty\left(\prod_{\substack{k=1\\k\in s_n(i)}}^m \left|\frac{z_n^i-z_k^i}{1-\overline{z_k^i}z_n^i}\right| \right)\right),
\] 
so if the sequence $(z_n)_n$  was assumed to satisfy $\sum_{n=1}^\infty(1-\lVert z_n\rVert)<\infty$, considering the limit as $m \to \infty$ we would have that $\prod_{k\neq n}\rho_{c_0}(z_k,z_n)\geq \frac{1}{M}>0$ for every $n$. In particular, we have just proved the following: 


\begin{theorem}\label{Teo-Bc0}
    If $(z_n)_n$ is an interpolating sequence for $H^\infty(B_{c_0})$ such that $\sum_{n=1}^\infty(1-\lVert z_n\rVert)<\infty$, then 
    \[
    \delta:=\inf_n \prod_{k\neq n}\rho_{c_0}(z_k,z_n)>0.
    \]
\end{theorem}

\section{Interpolating sequences for $\mathscr{H}^{\infty}$}\label{Section-4}
\vspace{0.3cm}
It was shown in \cite{seip2009interpolation} that a bounded sequence in $\C_+$ is interpolating for $\mathscr{H}^{\infty}$ whenever it is interpolating for $H^\infty(\C_+)$, and thus the geometry of the bounded interpolating sequences for $\mathscr{H}^{\infty}$ is completely determined by Carleson's theorem. However, a general characterization of the unbounded interpolating sequences for $\mathscr{H}^{\infty}$ appears to be a difficult problem. 

Since for each sequence $(s_n)_n$ in $\C_+$ that is interpolating for $\mathscr{H}^{\infty}$ the sequence $(x_n)_n$ defined by $x_n:=(p_m^{-s_n})_{m}$ is interpolating for $H^\infty(B_{c_0})$, we have that a necessary condition for $(s_n)_n$ to be interpolating for $\mathscr{H}^{\infty}$ is that $\text{Re}(s_n)\to 0$. In fact, it was proven in \cite{GalindoLindstromMiralles2009} that every interpolating sequence $(x_n)_n$ for $H^\infty(B_{c_0})$ satisfies $\lVert x_n\rVert\to 1$, and thus since $\lVert (p_m^{-s_n})_{m}\rVert=2^{-\text{Re}(s_n)}$, we must have $\text{Re}(s_n)\to 0$\footnote{This necessary condition can also be deduced either from \cite{seip2009interpolation}, or from the description of $\rho_{\mathscr{H}^{\infty}}$ and the fact that interpolating sequences are separated for the Gleason distance.}. In particular, if we consider for instance the sequence $s_n:=1+in$, we have that $(s_n)_n$ is interpolating for $H^\infty(\C_+)$ but not $\mathscr{H}^{\infty}$. 

Looking for necessary conditions on unbounded interpolating sequences for $\mathscr{H}^{\infty}$, we first observe that since for every $0\leq x\leq 1$ it is $\frac{x}{2}\leq 1-2^{-x}\leq (\log2) x$, we have that $\sum_{n=1}^\infty(1-\lVert (p_m^{-s_n})_{m}\rVert)<\infty$ if and only if $\sum_{n=1}^\infty \text{Re}(s_n)<\infty$, and therefore the following corollary is a straightforward consequence of Theorem \ref{Teo-Bc0}:

\begin{corollary}
    If $(s_n)_n\subset\C_+$ is an interpolating sequence for $\mathscr{H}^{\infty}$ satisfying $\sum_{n=1}^\infty \text{Re}(s_n)<\infty$, then 
    \[
    \delta:=\inf_n \prod_{k\neq n}\rho_{\mathscr{H}^{\infty}}(s_k,s_n)>0.
    \]
\end{corollary}

We will now use Theorem \ref{Main Theorem} to show a general procedure of constructing interpolating sequences for $\mathscr{H}^{\infty}$. We start by considering a sequence of integers $(k_n)_n$ and a real number $\sigma_1>0$, and then we define recursively a sequence $(\sigma_n)_n$ of positive real numbers such that $\rho_\D(2^{-\sigma_n},2^{-\sigma_k})\geq 1-\frac{1}{(nk)^4}$ for every $n>k$. 

Letting $s_n:=\sigma_n+i\frac{2\pi k_n}{\log2}$ for each $n$, we claim that $(s_n)_n$ is interpolating for $\mathscr{H}^{\infty}$. In fact, by \eqref{Gleason} we have that for each $n\neq k$ it is 
\[
\rho_{\mathscr{H}^{\infty}}(s_n,s_k)\geq\rho_\D(2^{-s_n},2^{-s_k})=\rho_\D(2^{-\sigma_n},2^{-\sigma_k})\geq 1-\frac{1}{(nk)^4},
\] 
so if we let $\delta_n:=\prod_{k\neq n}\rho_{\mathscr{H}^{\infty}}(s_n,s_k)$, we have by Lemma \ref{lema-blaschke} that 
\[
\delta_n\geq \prod_{k\neq n}\left(1-\frac{1}{(nk)^4}\right)\geq 1- \sum_{k\neq n} \frac{1}{(nk)^4}.
\] 
Therefore, we have that 
\[
\sum_{n=1}^\infty(1-\delta_n)^{1/3}\leq \sum_{n=1}^\infty \left(\sum_{k\neq n} \frac{1}{(nk)^4}\right)^{1/3} \leq \left(\frac{\pi^4}{90}\right)^{1/3}\sum_{n=1}^\infty\frac{1}{n^{4/3}}<\infty,
\] 
and we conclude by Theorem \ref{Main Theorem} that $(s_n)_n$ is interpolating for $\mathscr{H}^{\infty}$.
\vspace{0.3cm}
\subsection{$\mathscr{H}^{\infty}$ versus $H^\infty(\C_+)$ interpolation} 
\vspace{0.3cm}
Whether the interpolating sequences for $H^\infty(\C_+)$ and $\mathscr{H}^{\infty}$ coincide for any larger class than that of the bounded sequences in $\C_+$ has not been addressed yet. We will now show that for any sequence of positive real numbers $(\sigma_n)_n$ convergent to $0$ such that $\limsup_n\frac{\sigma_{n+1}}{\sigma_n}=1$, there is a sequence $(s_n)_n$ with $\text{Re} (s_n)=\sigma_n$ that is interpolating for $H^\infty(\C_+)$ but is not interpolating for $\mathscr{H}^{\infty}$. A crucial result for our construction that we now recall is Kronecker's theorem on Diophantine approximation \cite[Ch. XXIII]{HardyWright2008}, which as observed by Bohr implies that for every $k$, the mapping $s\mapsto(2^{-s},\dots, p_k^{-s})$ from $\R^+$ to $\T^k$ has dense range: 

\begin{theorem}[Kronecker]
    If $\vartheta_1,\dots,\vartheta_k$ are linearly independent real numbers and we are given any real numbers $\alpha_1,\dots,\alpha_k$, and positive real numbers $T$ and $\varepsilon$, then there are integers $p_1,\dots, p_k$ and a real number $t>T$ such that 
    \[
    |t\vartheta_i-p_i-\alpha_i|<\varepsilon \quad \text{for } i=1,\dots,k.
    \]
\end{theorem} 

Consider then a decreasing sequence of positive real numbers $(\sigma_n)_n$ that tends to $0$ and satisfies $\limsup_n\frac{\sigma_{n+1}}{\sigma_n}=1$, and define $s_1:=\sigma_1$. Since $\sup_{\alpha\in \R}\rho_\D(p^{-(\sigma_2+i\alpha)},p^{-\sigma_1})\to 0$ as $p\to\infty$, there is a positive integer $k_1$ such that for every prime number $p>k_1$ and every real number $\alpha$ it is 
\[
\rho_\D(p^{-(\sigma_2+i\alpha)},p^{-\sigma_1})<\frac{\rho_\D(2^{-\sigma_2},2^{-\sigma_1})}{2}.
\] 
Moreover, since the pseudohyperbolic distance in $\D$ is continuous, there is an $\varepsilon>0$ such that $|1-e^{i\alpha}|<\varepsilon$ implies that 
\[
|\rho_\D(p^{-(\sigma_2+i\alpha)},p^{-\sigma_1})-\rho_\D(p^{-\sigma_2},p^{-\sigma_1})|<\frac{\rho_\D(p^{-\sigma_2},p^{-\sigma_1})}{3}
\] 
for every prime number $p\leq k_1$. If we then let $T_1>0$ be such that $\rho_\D(\sigma_2+i\alpha,\sigma_1)\geq1-\frac{1}{2^3}$ for every $a> T_1$, we may apply Kronecker's theorem to find a real number $\alpha_2>T_1$ such that $|e^{-i\alpha_2 \log p}-1|<\varepsilon$ for every prime number $p\leq k_1$. Defining $s_2:=\sigma_2+i\alpha_2$, we have by \eqref{Gleason} that for some prime number $p_i\leq k_1$ it is 
\[
|\rho_{\mathscr{H}^{\infty}}(s_1,s_2)-\rho_\D(p_i^{-\sigma_2},p_i^{-\sigma_1})|<\frac{\rho_\D(p_i^{-\sigma_2},p_i^{-\sigma_1})}{3}.
\]  

If $s_1,\dots,s_n$ have already been defined, where $s_k=\sigma_k+i\alpha_k$ for some $\alpha_k$, we again have that since $\sup_{\alpha\in \R}\rho_\D(p^{-(\sigma_{n+1}+i\alpha)},p^{-s_n})\to 0$ as $p\to\infty$, there is a positive integer $k_n$ such that for every prime number $p>k_n$ and every real number $\alpha$ it is 
\[
\rho_\D(p^{-(\sigma_{n+1}+i\alpha)},p^{-s_n})<\frac{\rho_\D(2^{-\sigma_{n+1}},2^{-\sigma_n})}{2}.
\] 
Considering $\varepsilon>0$ such that 
\[
|\rho_\D(p^{-(\sigma_{n+1}+i(\alpha_n+\alpha))},p^{-s_n})-\rho_\D(p^{-(\sigma_{n+1}+i\alpha_n)},p^{-s_n})|<\frac{\rho_\D(p^{-(\sigma_{n+1}+i\alpha_n)},p^{-s_n})}{3}
\] 
for every prime number $p\leq k_n$ whenever $|1-e^{i\alpha}|<\varepsilon$, and $T_n>0$ such that for every $\alpha>T_n$ it is 
\[
\rho_\D(\sigma_{n+1}+i\alpha,s_k)\geq1-\frac{1}{2^{|n-k|+3}}\quad \text{for } k=1,\dots,n
\] 
we get by Kronecker's theorem an $\alpha_{n+1}>T_n$ such that $|e^{-i\alpha_{n+1} \log p}-1|<\varepsilon$ for every prime number $p\leq k_n$. If we then define $s_{n+1}:=\sigma_{n+1}+i(\alpha_{n+1}+\alpha_n)$, we observe that by construction there must be a prime number $p_j\leq k_n$ such that 
\[
|\rho_{\mathscr{H}^{\infty}}(s_n,s_{n+1})-\rho_\D(p_j^{-\sigma_{n+1}},p_j^{-\sigma_n})|<\frac{\rho_\D(p_j^{-\sigma_{n+1}},p_j^{-\sigma_n})}{3},
\] 
and we claim that the recursively defined sequence $(s_n)_n$ is interpolating for $H^\infty(\C_+)$ but not for $\mathscr{H}^{\infty}$. In fact, since $\rho_\D(s_n,s_k)\geq1-\frac{1}{2^{|n-k|+2}}$ whenever $n\neq k$, we have for every $n$ that 
\[
\sum_{k\neq n}(1-\rho_\D(s_n,s_k))\leq \sum_{k\neq n} \frac{1}{2^{|n-k|+2}}< 2\sum_{k=1}^\infty \frac{1}{2^{k+2}}=\frac{1}{2},
\] 
and therefore by Lemma \ref{lema-blaschke} it is  $\inf_n\prod_{k\neq n}\rho_\D(s_n,s_k)\geq \frac{1}{2}$, and thus Carleson's theorem implies that $(s_n)_n$ is interpolating for $H^\infty(\C_+)$. To see why $(s_n)_n$ cannot be interpolating for $\mathscr{H}^{\infty}$, observe that by construction we have that for every $n$ it is 

\[
\rho_{\mathscr{H}^{\infty}}(s_n,s_{n+1})<\frac{4}{3}\rho_\D(2^{-\sigma_{n+1}},2^{-\sigma_n})=\frac{4}{3}\frac{2^{\sigma_n}-2^{\sigma_{n+1}}}{2^{\sigma_n+\sigma_{n+1}}-1},
\] 
and since by the Mean Value theorem there is a $c\in(\sigma_{n+1},\sigma_{n})$ such that $2^{\sigma_n}-2^{\sigma_{n+1}}=2^c(\sigma_{n}-\sigma_{n+1})$,  we have in particular that 
\[
\rho_{\mathscr{H}^{\infty}}(s_n,s_{n+1})<\frac{4}{3}\frac{2^c\log2(\sigma_{n}-\sigma_{n+1})}{2^{\sigma_n+\sigma_{n+1}}-1}\leq \left(1-\frac{\sigma_{n+1}}{\sigma_n}\right)\frac{4}{3}\frac{2^{\sigma_n}\sigma_n\log2}{2^{\sigma_n+\sigma_{n+1}}-1}.
\] 
Now, taking into account that $\lim_{x\to 0}\frac{2^x-1}{x}=\log2$, we have that 
\[
\lim_{n\to\infty}\frac{2^{\sigma_n}\sigma_n\log2}{2^{\sigma_n+\sigma_{n+1}}-1}=\lim_{n\to\infty}\frac{\sigma_n}{\sigma_n+\sigma_{n+1}}=\lim_{n\to\infty}\frac{1}{1+(\sigma_{n+1}/\sigma_n)},
\] 
and therefore the hypothesis $\limsup_n\frac{\sigma_{n+1}}{\sigma_n}=1$ implies  $\inf_{n}\rho_{\mathscr{H}^{\infty}}(s_n,s_{n+1})=0$, so $(s_n)_n$ is not $\rho_{\mathscr{H}^{\infty}}$-separated, and in particular it cannot be interpolating for $\mathscr{H}^{\infty}$.

\section{Interpolating Sequences For Bidual Spaces}\label{Section-5}
\vspace{0.3cm}
Regarding interpolating sequences for the second dual of a uniform algebra $A$, Theorem \ref{Main-teo-bidual} provides only a sufficient condition for sequences in $M_A$ to be interpolating for $A^{**}$, so now we address the problem of characterizing the interpolating sequences in $M_{A^{**}}$. Motivated by Hoffman's classical characterization of interpolating sequences in the Shilov boundary of $H^\infty(\mathbb{D})$, we establish that a sequence in the Shilov boundary of the second dual of a uniform algebra $A$ is interpolating for $A^{**}$ if and only if it is discrete with respect to the $w^*$-topology.
We choose to work in a slightly more general frame, considering unital commutative Banach algebras whose Gelfand transform is an isomorphic embedding rather than uniform algebras. We call such Banach algebras \emph{uniformizable}, since they are isomorphic to uniform algebras. In particular, we have that the second dual of a uniformizable algebra $A$ is also uniformizable. 

The main theorem of this section is then the following: 

\begin{theorem}\label{int-bid}
    Let $A$ be a uniformizable Banach algebra. Then a sequence $(x_n)_n$ in $\partial_{A^{**}}$ is interpolating for $A^{**}$ if and only if it is discrete for the $w^*$-topology.
\end{theorem}

To prove the theorem, we need a topological result about the Shilov boundary of $A^{**}$, so first we recall some definitions. A compact space $S$ is said to be \emph{hyperstonean} if disjoint open sets have disjoint closures, and the union of the supports of all normal measures on $S$ is everywhere dense. Examples of hyperstonean spaces include the Stone-\v{C}ech compactification of any discrete space, the maximal ideal space of $L^\infty(\T)$, and any compact space $K$ for which $C(K)$ is isomorphic to a dual space. Moreover, for any compact space $K$ we have that $C(K)^{**}$ is isometrically isomorphic to $C(S)$ for some hyperstonean space $S$. The property of being hyperstonean is not hereditary even for closed subspaces, but there is a wider class of topological spaces called \emph{F-spaces}, which consists of those compact spaces for which any pair of disjoint $F_\sigma$ open sets has disjoint closures, that was proven to be hereditary for closed subspaces in \cite{Seever1968}. 


\begin{proposition}\label{hyperstonean-bid}
    If $A$ is a uniformizable Banach algebra with bidual space $A^{**}$, then the Shilov boundary of $A^{**}$ is a hyperstonean space.
\end{proposition}

\begin{proof}
    Since by hypothesis the Gelfand transform $\kappa \colon A \to C(M_A)$ is a bounded below homomorphism, so is its second adjoint $\kappa^{**} \colon A^{**} \to C(M_A)^{**}$. By the previous comments, there is an isometric isomorphism $i \colon C(M_A)^{**} \to C(K)$ where $K$ is a hyperstonean space, so considering the composition $i\circ \kappa^{**} \colon A^{**} \to C(K)$ we have that its adjoint is also a homomorphism which maps the spectrum of $C(K)$, which is homeomorphic to $K$, into the hyperstonean subspace $S:=(i\circ \kappa^{**})^*(K)$ of $M_{A^{**}}$. 
    
    In order to prove that $S$ is a boundary for $A^{**}$, we observe that otherwise there would exist a function $f$ in $A^{**}$ such that $\lVert\hat{f}\rVert=1$ while $\eta:=\sup_{x\in S}|\hat{f}(x)|<1$. Now, since $i\circ \kappa^{**} \colon A^{**} \to C(K)$ is bounded below, there exists a constant $C\geq 1$ such that $\lVert g\rVert\leq C\lVert i\circ \kappa^{**}(g)\rVert_\infty$ for every $g\in A^{**}$, and considering a positive integer $n$ such that $\eta^n<C^{-1}$ we would have that the function $h:=f^n$ would belong to $A^{**}$ and would satisfy $\lVert h\rVert\geq\lVert\hat{h}\rVert=\lVert(\hat{f})^n\rVert=\lVert\hat{f}\rVert=1$. However, this would imply that 
    \[
    1\leq \lVert h\rVert\leq C\lVert i\circ \kappa^{**}(h)\rVert_\infty=C \sup_{x\in S}|\hat{h}(x)|=C\eta^n <1,
    \] 
    and this contradiction proves that $S$ is a boundary. 

    Since $A^{**}$ is also uniformizable, its Gelfand transform is also an isomorphic embedding, and because $S$ is a boundary for $A^{**}$ we have that $\widehat{A^{**}}$ is isometrically isomorphic to $\widehat{A^{**}}|_S\leq C(S)$. Therefore, if we consider any proper closed set $X\subsetneq S$, we have that since $S$ is in particular zero-dimensional there would be a nonempty clopen set $U$ contained in $S\setminus X$, and then by Shilov's Idempotent Theorem \cite[p. 88]{Gamelin1969} there would exist a function $\hat{f}$ in $\widehat{A^{**}}$ such that $\hat{f}|_S\in \widehat{A^{**}}|_S$ is idempotent and satisfies $\hat{f}|_S(U)=1$ while $\hat{f}|_S(S\setminus U)=0$. The function $\hat{f}$ would then satisfy $\lVert\hat{f}\rVert>\sup_{x\in X}|\hat{f}(x)|$, so $X$ could not be a boundary and we conclude that $S=\partial_{A^{**}}$, from where the theorem follows.
\end{proof}

As a consequence, we observe that although the spectrum of $A$ embeds homeomorphically in $M_{A^{**}}$, the Shilov boundary of $A$ need not embed as a closed subset of $\partial_{A^{**}}$, because this would imply that $\partial_A$ is always a compact $F$-space and considering for instance the disk algebra $A(\D)$ this is not true. 

We now prove the main result of this section:

\begin{proof}[Proof of Theorem \ref{int-bid}]
    If is easily seen that if $(x_n)_n$ is interpolating for $A^{**}$ then it must be discrete, so we only prove the other implication. Suppose $(x_n)_n$ is a discrete sequence in $\partial_{A^{**}}$, the Shilov boundary of $A^{**}$. Since $A$ is a uniformizable Banach algebra, we have by Proposition \ref{hyperstonean-bid} that $\partial_{A^{**}}$ is a hyperstonean subspace $\partial_{A^{**}}\subseteq M_{A^{**}}$, and moreover $\widehat{A^{**}}$ is isometrically isomorphic to $\widehat{A^{**}}|_{\partial_{A^{**}}}\leq C(\partial_{A^{**}})$. Now, given two disjoint subsets of natural numbers $P,Q\subseteq\N$, there are disjoint open sets $U,V\subset \partial_{A^{**}}$ such that $\{x_n\}_{n\in P}\subseteq U$ and $\{x_n\}_{n\in Q}\subseteq V$, and thus since $\partial_{A^{**}}$ is hyperstonean, we have that 
    \[
    \overline{\{x_n\}}_{n\in P}\cap \overline{\{x_n\}}_{n\in Q}\subseteq \overline{U}\cap \overline{V}=\emptyset.
    \]
    
    By Shilov's Idempotent Theorem, there is a function $\hat{f}$ in $\widehat{A^{**}}$ such that $\hat{f}|_{\partial_{A^{**}}}\in \widehat{A^{**}}|_{\partial_{A^{**}}}$ is idempotent and satisfies $\hat{f}|_{\partial_{A^{**}}}(\overline{U})=1$ while $\hat{f}|_{\partial_{A^{**}}}(\overline{V})=0$. We thus have that if $T \colon A^{**}\rightarrow \ell_\infty$ is the interpolating operator associated with the sequence $(x_n)_n$, then $T(A^{**})$ contains the space of finite-valued sequences, which by a well-known theorem of Grothendieck is a barrelled space, and therefore by \cite{BennettKalton1973} we conclude that the sequence $(x_n)_n$ is interpolating for $A^{**}$. 
\end{proof}

\section{Remarks on Carleson's generalized problem}\label{Section-6}
\vspace{0.3cm}
We conclude by discussing a possible approach to Carleson's generalized problem through the study of Carleson measures in the polydisc, as a positive solution to this problem would yield a significant strengthening of Theorem \ref{Main Theorem}.

Let $m$ denote the Lebesgue measure in $\D^N$. Consider for each $z=re^{i\theta}$ in $\D$ the set $I_z:=\{e^{i\sigma}:|\sigma-\theta|<1-r\}$, and if $U\subseteq \T^N$ is an open connected,   let  $S(U):=\{(z_1,\dots,z_N)\in\D^N : I_{z_1}\times\dots\times I_{z_N}\subseteq U \}$. Following \cite{Chang1979}, we say that $\mu$ is a Carleson measure in $\D^N$ if there is a positive constant $C>0$ such that for every open connected set $U\subseteq \T^N$ it holds
\begin{equation}\label{med-carl-poly}
    \mu(S(U))\leq C m(U). 
\end{equation} 
The infimum of all $C>0$ such that \eqref{med-carl-poly} holds for every open connected set $U\subseteq \T^N$ is called \emph{Carleson's intensity} of $\mu$. 
Carleson measures arise naturally from the theory of interpolating sequences. In fact, given a sequence $(x_n)_n$ in $\D^N$, associating to each  $x_n=(w_1^n,w_2^n,\dots,w_N^n)$ the rectangle $R_n$ in $\T^N$ which is centered at the point $(w_1^n/|w_1^n|,w_2^n/|w_2^n|,\dots,w_N^n/|w_N^n|)$ and has side lengths $2(1-|w_1^n|),\dots,2(1-|w_N^n|)$, it follows from a theorem of Varopoulos \cite{Varopoulos1972} that if $(x_n)_n$ is interpolating for $H^\infty(\D^N)$, then $\mu:=\sum_{n=1}^\infty|R_n|\delta_{x_n}$ is a Carleson measure on $\D^N$. Moreover, the authors of \cite{BerndtssonChangLin1987} proved that if $(x_n)_n$ satisfies \eqref{Carl-gen}, then Carleson's intensity of $\mu$ is independent of the dimension $N$.

Our approach to Carleson's generalized problem is motivated by the observation made in \cite{BerndtssonChangLin1987}  that a positive answer to the problem would follow if every sequence in $\D^N$ satisfying \eqref{Carl-gen} were interpolating for $H^\infty(\D^N)$ with interpolating constant bounded by some constant independent of $N$. 

The following proposition reduces the problem to a quantitative comparison between interpolation constants and Carleson intensities. In particular, if for each $(x_n)_n$ we denote by $M(\{x_n\}_n)$ its constant of interpolation, and we let $C(\{x_n\}_n)$ be the Carleson intensity of the measure $\mu:=\sum_{n=1}^\infty|R_n|\delta_{x_n}$, we have:

\begin{proposition}\label{reduccion-carl-gen}
   Suppose that 
   \begin{equation}\label{prop}
       \sup\left\{ \frac{M(\{x_n\}_n)}{C(\{x_n\}_n)}: \{x_n\}_n \text{ is int. for } H^\infty(\D^N)\text{ for some } N\in\N\right\}<\infty.
   \end{equation} 
   Then for any dual uniform algebra $A=X^*$ and any sequence $(x_n)_n\subset M_A\cap X$ satisfying 
   \[
   \inf_n \prod_{k\neq n}\rho_A(x_n,x_k)\geq \delta>0
   \] 
   we would have that $(x_n)_n$ is interpolating for $A$.
\end{proposition}

\begin{proof}
    We claim that under our hypothesis the interpolation constant of a sequence $(x_n)_n\subset \D^N$ satisfying the condition $\inf_{n\in\N}\prod_{k\neq n}\rho_{\D^N}(x_n,x_k):=\delta>0$ could be bounded above by a constant depending only on $\delta$ and not on the dimension $N$. Otherwise, we could find an increasing sequence of natural numbers $(N_k)_k$ and sequences $(x_n^{N_k})_n\subset \D^{N_k}$ satisfying 
    \[
    \inf_{j}\prod_{\substack{i=1\\ i\neq j}}^\infty\rho_{\D^{N_k}}(x_i^{N_k},x_j^{N_k})\geq\delta>0\quad \text{for every } k\in\N 
    \] 
    while $\lim_{k\to\infty}M(\{x_n^{N_k}\}_n)=\infty$. However this is not possible, as \eqref{prop} would imply that $\lim_{k\to\infty}C(\{x_n^{N_k}\}_n)=\infty$, contradicting \cite[Proposition 6]{BerndtssonChangLin1987}. 
    
    Therefore, given $\alpha=(\alpha_n)_n\in\ell_\infty$, we have that for every natural number $N$ there would exist a function $f_N\in A$ with $\lVert f_N\rVert\leq M$ satisfying 
    \[
    f_N(x_n)=\alpha_n\quad \text{for } n=1,\dots,N,
    \] 
    and thus Alaoglu's Theorem guarantees that $(f_n)_n\subset A$ has a subnet $(f_{n_i})_i$ $w^*$-converging to a function $f\in A$ that necessarily satisfies $\lVert f \rVert\leq M$ and 
    \[
    f(x_n)=\lim_i f_{n_i}(x_n)=\alpha_n\quad \text{for every } n.
    \]
\end{proof}

\section*{Acknowledgment}

\end{document}